\documentclass[12pt]{amsart}
\usepackage{amsmath}
\usepackage{amssymb}
\usepackage{amscd}
\usepackage{verbatim}

\def\NZQ{\mathbb}               
\def\NN{{\NZQ N}}

\def\ZZ{{\NZQ Z}}
\def\RR{{\NZQ R}}

\newtheorem{Theorem}{Theorem}[section]

\newtheorem{Corollary}[Theorem]{Corollary}

\newtheorem{Remark}[Theorem]{Remark}

\newtheorem{Definition}[Theorem]{Definition}

\let\epsilon\varepsilon
\let\phi=\varphi
\let\kappa=\varkappa

\begin{document}

\title{Dependent Artin-Schreier defect extensions and strong monomialization}

\author{Samar ElHitti} \address{Department of Mathematics, New York City
College of Technology-Cuny, 300 Jay Street, Brooklyn, NY 11201, U.S.A.} \email{selhitti@citytech.cuny.edu}
\author {Laura Ghezzi} \address{Department of Mathematics, New York City
College of Technology-Cuny, 300 Jay Street, Brooklyn, NY 11201, U.S.A.} \email{lghezzi@citytech.cuny.edu}

\thanks{The authors would like to thank Franz-Viktor Kuhlmann for motivating this work, and S. Dale Cutkosky for many helpful conversations regarding the material in this paper.\\
The first author was supported by a PSC-CUNY Award, jointly funded by The Professional Staff Congress and The City University of New York.\\
The second author was partially supported by the Fellowship Leave from the New York City College of Technology-CUNY (Fall 2014 - Spring 2015).}

\begin{abstract}
\noindent
In this paper we affirmatively answer a question posed by F.-V. Kuhlmann. We show that the first Artin-Schreier defect extension in Cutkosky and Piltant's counter-example to strong monomialization is a dependent extension. Our main tool is the use of generating sequences of valuations.
\end{abstract}

\maketitle

\section{Introduction}\label{intro}

Let $K$ be a field and let $\nu$ be a valuation on $K$. The value group of $(K,\nu)$ is denoted by $\nu K$, and its residue field by $K\nu$. The value of an element $a$ is denoted by $\nu(a)$.
Let $L$ be a finite field extension of $K$. If the extension of $\nu$ from $K$ to $L$ is unique (denoted by $(L|K,\nu)$), the Lemma of Ostrowski says that
\begin{equation}\label{Os}[L:K]=(\nu L: \nu K)\cdot [L\nu : K \nu]\cdot p^s,\end{equation} with $s\geq 0$, where $p$ is the characteristic of $K\nu$ if it is positive, and $p=1$ if $K\nu$ has characteristic zero. The factor $p^s$ is called the {\it defect} (or ramification deficiency) of the extension $(L|K,\nu)$. If $p^s=1$ we call $(L|K,\nu)$ a defectless extension; otherwise we call it a defect extension. Note that $(L|K,\nu)$ is always defectless if char $K\nu=0$. The defect plays a key role in deep open problems in positive characteristic, such as local uniformization (the local form of resolution of singularities). Indeed the existence of the defect makes the problem of local uniformization much harder, as can be seen in \cite{KK1} and \cite{KK2}. We refer the reader to \cite{Ku10} for an excellent overview of the valuation theoretical phenomenon of the defect.

\medskip

Particular types of extensions that are ubiquitous in this set up are {\it Artin-Schreier extensions}. From now on we assume that all fields have characteristic $p>0$. An Artin-Schreier extension of $K$ is an extension of degree $p$ generated over $K$ by a root $\theta$ of a polynomial of the form $X^p-X-a$ with $a\in K$. Such $\theta$ is called an Artin-Schreier generator of $L=K(\theta)$ over $K$. An extension of degree $p$ of a field of characteristic $p$ is a Galois extension if and only if it is an Artin-Schreier extension \cite[Theorem VI.6.4]{L}. We assume that $L|K$ has defect. There is a unique extension of $\nu$ from $K$ to $L$ (see \cite[Lemma 2.31]{Ku19}). It follows from equation (\ref{Os}) that $L$ is an immediate extension of $K$; that is, $\nu L=\nu K$ and $L\nu=K\nu$.

\medskip

In \cite{Ku19} Kuhlmann classifies Artin-Schreier defect extensions
and gives the motivation for such a classification. If an Artin-Schreier defect extension is derived from a purely inseparable defect extension of degree $p$,
then we call it a {\it dependent }Artin-Schreier defect extension. If it cannot be derived in this way, then we call it an {\it independent} Artin-Schreier defect extension. The precise definitions are given in Section \ref{criteria}.

Kuhlmann (see \cite[Section 6.1]{Ku10}) raised the question of which of the Artin-Schreier defect extension are more ``harmful'', the dependent or the independent ones? He points out that there are some indications that the dependent ones are more harmful, for instance based on Temkin's work \cite{Tem}.

\medskip

In this paper we are interested in Cutkosky and Piltant's counterexample to strong monomialization.
In \cite[Theorem 7.38]{CP} they give an example of an extension $R\subset S$ of two-dimensional algebraic regular local rings over a field ${\bf k}$ of positive characteristic, and a valuation on the rational function field $Q(R)$. The extension $Q(S)|Q(R)$ is a tower of two Artin-Schreier defect extensions, such that strong monomialization (in the sense of \cite[Theorem 4.8]{CP}) does not hold for $R\subset S$. We will recall the example in more details in Section \ref{SM}.

Strong monomialization (\cite[Theorem 4.8]{CP} in characteristic zero and any dimension) is a very important result, as it implies simultaneous resolution, which is very useful for applications to local uniformization (see \cite{CP}, \cite{Ab2} for details). Cutkosky and Piltant further prove that strong monomialization holds over algebraically closed fields, in dimension two and positive characteristic, provided the field extension is defectless (\cite[Theorem 7.35]{CP}). Recently Cutkosky gave a counterexample to local and weak local monomialization in positive characteristic and any dimension greater than or equal to two (\cite[Theorem 1.4]{C}).

\medskip

Kuhlmann asked whether there is at least one dependent extension in the counterexample to strong monomialization -- a further indication that dependent extensions are more harmful.

We prove in Theorem \ref{maintheorem} that the first extension in the tower constructed in \cite[Theorem 7.38]{CP} is an Artin-Schreier dependent extension.

It is still an open problem whether strong monomialization holds when independent Artin-Schreier defect extensions are involved.

\medskip

We now state the layout of this paper.
In Section \ref{SM} we recall the counterexample to strong monomialization.
 In Section \ref{criteria} we recall the definition of dependent (and independent) Artin-Schreier defect extensions and establish that, in our context, it is enough to show that there exists $m \in \mathbb{N}_0$ such that for all $f \in K, \ \nu(\theta-f)< -\frac{1}{p^m}$. Finally, in Section \ref{main} we prove that the extension under consideration is dependent (Theorem \ref{maintheorem}). Since the valuation is determined by a generating sequence in the given ring $R$, the main point is to construct generating sequences in appropriate quadratic transforms $R_k$ of $R$ that allow to compute or bound values of certain key elements of $R_k$. We achieve the construction in Theorem \ref{genconstruction}, and we achieve the bound on the values of critical elements in $R_k$ in Theorem \ref{Theorem1}.

\section{Counterexample to strong monomialization}\label{SM}

In this section, we use the notation of \cite[Section 7.11] {CP}.
Let ${\bf k}$ be an algebraically closed field of characteristic $p>0$. Cutkosky and Piltant prove that there exists an inclusion $R\subset S$ of algebraic regular local rings of dimension $2$ over ${\bf k}$ such that $K^*|K=Q(S)|Q(R)$ is a tower of two Galois extensions of degree $p$, and a ${\bf k}$-valuation $\nu^*$ of $K^*$ with valuation ring $V^*$ such that
 $V^*|V$ has defect $p^2$ (where $V$ is the valuation ring of the restriction of $\nu^*$ to $K$).
 Let $R_r$ and $S_s$ be iterated quadratic transforms of $R$ and $S$ respectively, such that $R_r\subset S_s$. Then there exists a regular system of parameters $(u_r,v_r)$ of $R_r$ and a regular system of parameters $(x_s,y_s)$ of $S$ with such an expression:

\[
  \begin{array}{lll}
 u_r&=&\gamma_{r,s}x_s^p\\
 v_r&=&\delta_{r,s}y_s^p
 \end{array}
\]

where $\gamma_{r,s}, \delta_{r,s}$ are units in $S_s$. Therefore strong monomialization in the sense of \cite[Theorem 4.8]{CP} fails for the pair $R\subset S$ with respect to the extension of valuation rings $V^*|V$. See \cite[Theorem 7.38]{CP} for the full statement of their result.

\begin{Remark}{\rm Strong monomialization requires that the inclusion $R_r\subset S_s$ be given by
\[
  \begin{array}{lll}
 u_r&=&\gamma_{r,s}x_s^t\\
 v_r&=&\delta_{r,s}y_s
 \end{array}
\]

where $\gamma_{r,s}, \delta_{r,s}$ are units in $S_s$ and $t$ is a positive integer.}
\end{Remark}

To construct the extension $K^*|K$, let $c\ge 1$ be an integer such that $(p-1)|c$ and let $K^*={\bf k}(x,y)$ and $K={\bf k}(u,v)$ be two dimensional algebraic function fields, where

\[
  \begin{array}{lll}
 u&:=&x^p/(1-x^{p-1})\\
 v&:=&y^p-x^cy.
\end{array}
\]
Then $K^*|K$ is a finite separable extension of degree $p^2$. It is a tower of two degree $p$ Artin-Schreier extensions
 $$
 K\rightarrow K_1=K(x)\rightarrow K^*=K_1(y).
 $$

\medskip

We focus on the first extension $K(x)|K$. We have that $\frac{1}{x}$ is an Artin-Schreier generator with minimal polynomial $F(X)=X^p-X-\frac{1}{u}$.

\medskip

We recall that the field $K^*={\bf k}(x,y)$ has the valuation $\nu^{*}$ determined by the generating sequence
$Q_0=x$, $Q_1=y$, $Q_2=y^{p^2}-x$  and
$Q_{i+1}=Q_i^{p^2}-x^{p^{2i-2}}Q_{i-1}$ for $i\ge 2$ in $S={\bf k}[x,y]_{(x,y)}$ and the value group denoted by $\Gamma^* \simeq \displaystyle\bigcup_{i\geq 0}\frac{\mathbb{Z}}{p^i}$ (see \cite[Proposition 7.40]{CP}).

\medskip

The field $K={\bf k}(u,v)$ has the valuation $\nu$ obtained by restriction of $\nu^*$ to $K$. $\nu$ is determined by the generating sequence
$P_0=u$, $P_1=v$, $P_2=v^{p^2}-u$  and
$P_{i+1}=P_i^{p^2}-u^{p^{2i-2}}P_{i-1}$ for $i\ge 2$ in $R={\bf k}[u,v]_{(u,v)}$ (see \cite[Corollary 7.41]{CP}). We have that the value group $\Gamma=\Gamma^* \simeq \displaystyle\bigcup_{i\geq 0}\frac{\mathbb{Z}}{p^i}$. We denote the unique extension of $\nu$ to $K_1$ by $\nu_1$. Notice that $\nu_1$ is the restriction of $\nu^*$ to $K_1$.

\medskip

We prove in Theorem \ref{maintheorem} that $K(x)$ is a dependent Artin-Schreier defect extension of $K$. Our main tool is generating sequences, therefore we briefly recall the definition (as in \cite{S},  \cite[Section 7.5]{CP} or \cite[Section 2]{CV}).

\begin{Definition}\label{genseq} {\rm Let $R$ be an algebraic two dimensional regular local ring, let $K$ be the quotient field of $R$, and let $\nu$ be a valuation of $K$ centered in $R$.
Let $({\nu K})_+ = \nu(R \backslash \{0\})$ be the semigroup of $\nu K$
consisting of the values of nonzero elements of $R$. For $\gamma \in
({\nu K})_+$, let $I_{\gamma}=\{f\in R \mid \ \nu(f)\geq \gamma\}$. A
(possibly infinite) sequence $\{ P_i \}$ of elements of $R$ is a
{\it generating sequence} of $\nu$ if for every $\gamma
\in ({\nu K})_+$ the ideal $I_\gamma$ is generated by the set
$$\left\{\prod_{i}{P_i}^{a_i}\mid \ a_i\in \mathbb{N}_0,\
\sum_{i} a_i \nu(P_i)\geq \gamma\right\}.$$ }
\end{Definition}

\section{Artin-Schreier defect extensions and criteria for dependence}\label{criteria}

In this section we recall the definition of dependent and independent Artin-Schreier defect extensions as in \cite[Section 4.1]{Ku19}, or \cite[Section 6]{Ku10}. Then we assemble results from \cite{Ku10} and \cite{Ku19} to obtain a criterion to detect if an Artin-Schreier defect extension is dependent in the context of the example under study recalled in Section \ref{SM}.

\begin{Definition}\label{def}{\rm Let $(K(\theta)|K,\nu)$ be an Artin-Schreier defect extension of valued fields of characteristic $p>0$ with Artin-Schreier generator $\theta$ such that $\theta^p-\theta\in K$. We call
$(K(\theta)|K,\nu)$ a dependent extension if there exists an immediate purely inseparable extension $(K(\eta)|K,\nu)$ of degree $p$ such that for all $c\in K$, $\nu(\theta-c)=\nu(\eta-c)$.
An Artin-Schreier defect extension which is not dependent is called independent.}
\end{Definition}

See for instance \cite[Example 3.17]{Ku10} for an example of a dependent extension, and \cite[Example 3.12]{Ku10} for an independent one. Several other examples of dependent and independent Artin-Schreier defect extensions are given in \cite[Section 4.6]{Ku19}. We remark that
by \cite[Lemma 4.1]{Ku19}, Definition \ref{def} does not depend on the choice of the Artin-Schreier generator $\theta$.

\medskip

Now assume that the value group $\nu K$ is Archimedian; that is, it can be embedded by an order preserving isomorphism as a subgroup of $\RR$.

Since $K(\theta)$ is an immediate extension of $K$ and $\nu(\theta-c)<0$ for all $c\in K$ by  \cite[Corollary 2.30]{Ku19}, we have that $\nu(\theta - K) \subseteq (\nu K)^{<0},$ where $\nu(\theta-K)=\{\nu(\theta-c)\mid c\in K\}$ and $(\nu K)^{<0}=\{\nu(c)\mid c\in K , \ \nu(c)<0\}$.

By \cite[Theorem 6.1]{Ku10} and the discussion following it, we have that the extension $(K(\theta)|K, \nu)$ is independent if and only if $$\nu(\theta - K)=(\nu K)^{<0}.$$

\begin{Theorem} \label{thcriteria} Let $(K(\theta)|K,\nu)$ be an Artin-Schreier defect extension with $\theta^p-\theta \in K$. Assume that $\nu K=\displaystyle\bigcup_{i\geq 0}\dfrac{\mathbb{Z}}{p^i}$. Then:
\item[1)] $K(\theta)|K$ is independent if and only if for all $m \in \mathbb{N}_0$ there exists $f_m \in K$ such that  $\nu(\theta-f_m)\geq -\dfrac{1}{p^m}.$
\item[2)] $K(\theta)|K$ is dependent if and only if there exists $m \in \mathbb{N}_0$ such that for all $f \in K, \nu(\theta-f)< -\dfrac{1}{p^m}.$
\end{Theorem}

\begin{proof} By the above discussion $(K(\theta)|K,\nu)$ is independent if and only if $(\nu K)^{<0}\subseteq \nu(\theta - K)$. Let  $-\dfrac{1}{p^m}\in(\nu K)^{<0}$, where  $m \in \mathbb{N}_0$. Let $f_m \in K$ be such that  $\nu(\theta-f_m)\geq -\dfrac{1}{p^m}$.
Since $\nu(\theta - K)$ is an initial segment of $\nu K$ (\cite[Lemma 1.1]{Ku10}), we have that $-\dfrac{1}{p^m}\in \nu(\theta - K)$. This proves 1). \end{proof}

\section{Main result}\label{main}

In this section, setup and notation are as in \cite[Theorem 7.38]{CP} and Section \ref{SM}.
Recall that the algebraic function field $K={\bf k}(u,v)$ has the valuation $\nu$ determined by the generating sequence
$P_0=u$, $P_1=v$, $P_2=v^{p^2}-u$  and
$P_{i+1}=P_i^{p^2}-u^{p^{2i-2}}P_{i-1}$ for $i\ge 2$ in $R={\bf k}[u,v]_{(u,v)}$. We denote the unique extension of $\nu$ to $K_1$ by $\nu_1$.

\medskip

We normalize $\nu$ so that
$\nu(u)=1$, from which we obtain the formulas $\nu(P_0)=1$ and
$$
\nu(P_i)=\sum_{j=0}^{i-1}p^{4j-2i}\mbox{ for }i\ge 1.
$$
Let
$$
R=R_0\rightarrow R_1\rightarrow R_2\rightarrow \cdots\rightarrow R_k\rightarrow \cdots
$$
be the sequence of iterated quadratic transforms along $\nu$, where $R_k$ are the local rings in the sequence of quadratic
transforms of $R$ along $\nu$ which are free; the divisor of $u$ in $R_k$ has a single irreducible component.
Let $m_k$ be the maximal ideal of $R_k$.

 We have that the valuation ring $V$ of $\nu$ is $V=\displaystyle\bigcup_{i=0}^{\infty}R_k$ (\cite{Ab1}, see also \cite[Corollary 7.41]{CP}).

\begin{Remark}\label{gs} {\rm Let $t_k\in R_k$ be irreducible such that $t_k=0$ is a local equation of the divisor of $u$ in $R_k$.
For $i\ge k+1$, let $n_{k,i}$ be the largest power of $t_k$ which divides $P_i$ in $R_k$.  By \cite{S}, the discussion
after  \cite[Definition 7.11]{CP} or \cite[Theorem 7.1]{CV}, we have that
$$
t_k,\frac{P_{k+1}}{t_k^{n_{k,k+1}}}
$$
is a regular system of parameters in $R_k$ and
$$
t_k,\frac{P_{k+1}}{t_k^{n_{k,k+1}}}, \frac{P_{k+2}}{t_k^{n_{k,k+2}}},\cdots, \frac{P_{k+j}}{t_k^{n_{k,k+j}}},\cdots
$$
is a generating sequence in $R_k$ for the valuation $\nu$. }
\end{Remark}

Next we explicitly construct generating sequences in the local rings $R_k$.

\begin{Theorem}\label{genconstruction} With notation as above, for $k\geq 0$ the local rings
$R_k$ have generating sequences
$$
u_k=P_{k,0}, \ v_k=P_{k,1}, \ P_{k,2},\cdots, P_{k,i}, \cdots
$$
which are defined recursively by the initial conditions
$$
u_0=u, \ v_0=v, \ P_{0,i}=P_i\mbox{ for }i\ge 0,
$$
and, for $k\ge 1$, by
\begin{equation}\label{eq2}
u_k=P_{k,0}=P_{k-1,1}, \ v_k=P_{k,1}=\frac{P_{k-1,2}}{u_k^{p^2}},
\end{equation}
and by
\begin{equation}\label{eq3}
P_{k,i}=\frac{P_{k-1,i+1}}{u_k^{p^{2i}}}\mbox{ for }i\ge 1.
\end{equation}
We thus have the formulas
\begin{equation}\label{eq1}
\nu(P_{k,0})=\frac{1}{p^{2k}},\ \nu(P_{k,1})=\frac{1}{p^{2k+2}}, \ \nu(P_{k,i})= \sum_{j=0}^{i-1}p^{4j-2i-2k}\mbox{ for }i\ge 1.
\end{equation}

\end{Theorem}

\begin{proof}
Define
$u_0=u, v_0=v$ and $P_{0,i}=P_i$ for $i\ge 0$.

We will inductively construct
$$
P_{k,0},P_{k,1},\ldots,P_{k,i},\ldots
$$
 in $R_k$ for $k\ge 1$ such that the following three conditions are satisfied:
\vskip .2truein

\begin{enumerate}

\item[C(1,k):] Set
$$
u_k=P_{k,0}, \ v_k=\frac{P_{k-1,2}}{u_k^{p^2}}.
$$
Then $u_k,v_k$ are regular parameters in $R_k$ such that
$u=u_k^{p^{2k}}\tau_k$,  where $\tau_k$ is a unit in $R_k$.

\vskip .2truein

\item[C(2,k):] Set
$$
P_{k,i}=\frac{P_{k-1,i+1}}{u_k^{p^{2i}}}
$$
for $i\ge 1$. Then
$$
u_k=P_{k,0}, \ v_k=P_{k,1},\ P_{k,2},\ldots, P_{k,i}, \ldots
$$
is a generating sequence for $\nu$ in $R_k$.

\vskip .2truein

\item[C(3,k):] For $i\ge 2$, there exist units $\gamma_{k,i}\in R_k$ such that
$$
P_{k,i}=\left\{\begin{array}{ll}
v_k^{p^2}-\gamma_{k,2}u_k=P_{k,1}^{p^2}-\gamma_{k,2}P_{k,0}&\mbox{ if }i=2\\
P_{k,i-1}^{p^2}-\gamma_{k,i}u_k^{p^{2(i-2)}}P_{k,i-2}&\mbox{ if }i\ge 3
\end{array}\right.
$$
and $\gamma_{k,i}\equiv 1\mbox{ mod }m_k^2$.
\end{enumerate}

\vskip .2truein

First, from the sequence $\{P_{0,i}\}$, we construct the sequence $\{P_{1,i}\}$ which satisfies C(1,1), C(2,1) and C(3,1), using a simplification of the following argument for constructing $\{P_{k+1,i}\}$ for $k\ge 1$. We note that the simplification in this first step arises from the fact that the units $\gamma_{k,0}\in R_0$ are all equal to $1$.

\vskip .2truein

Now assume that $k\ge 1$ and that we have constructed the sequence $\{P_{j,i}\}$ for $j\le k$, such that C(1,j), C(2,j) and C(3,j) hold for $j\le k$.
We construct the sequence $\{P_{k+1,i}\}$ such that C(1,k+1), C(2,k+1) and C(3,k+1) hold.

Define $u_{k+1}, s_{k+1}$ by
\begin{equation}\label{eq12}
u_k=u_{k+1}^{p^2}(s_{k+1}+1), \ v_k=u_{k+1}.
\end{equation}
Equation (\ref{eq12}) defines a sequence of quadratic transforms
$$
R_k\rightarrow T_{k+1}:= R_k[s_{k+1}]_{(u_{k+1},s_{k+1})}.
$$
In the factorization of $R_k\rightarrow T_{k+1}$ by quadratic transforms, $T_{k+1}$ is the first local ring in which the divisor of $u_k$ (and of $u$)
has a single irreducible component.
Since
$$
\nu(v_k^{p^2}-u_k)=\nu(v_k^{p^2}-\gamma_{k,2}u_k+\gamma_{k,2}u_k-u_k)>\nu(u_k)=\nu(v_k^{p^2})
$$
by C(2,k) and C(3,k), we have that $\nu(s_{k+1})>0$. Therefore $\nu$ dominates $T_{k+1}$ and  $R_{k+1}=T_{k+1}$. Thus
$u_{k+1}, s_{k+1}$ are regular parameters in $R_{k+1}$. By C(1,k), we have that
$$
u=u_k^{p^{2k}}\tau_k=u_{k+1}^{p^{2(k+1)}}\tau_{k+1},
$$
where
$$
\tau_{k+1}=(s_{k+1}^{p^{2k}}+1)\tau_k.
$$

We have
$$
\begin{array}{lll}
P_{k,2}&=& v_k^{p^2}-\gamma_{k,2}u_k\\
&=& u_{k+1}^{p^2}-\gamma_{k,2}u_{k+1}^{p^2}(s_{k+1}+1)\\
&=& u_{k+1}^{p^2}((1-\gamma_{k,2})-\gamma_{k,2}s_{k+1}).
\end{array}
$$
Thus
$$
v_{k+1}=\frac{P_{k,2}}{u_{k+1}^{p^2}}\in R_{k+1}
$$
and $u_{k+1},v_{k+1}$ are regular parameters in $R_{k+1}$. In particular, C(1,k+1) holds.

\vskip .2truein

For $i\ge 2$, using C(2,k), C(3,k) and equation (\ref{eq12}), we have
$$
\begin{array}{lll}
u_{k+1}^{p^{2i}}P_{k+1,i}&=& P_{k,i+1}=P_{k,i}^{p^2}-\gamma_{k,i+1}u_k^{p^{2(i-1)}}P_{k,i-1}\\
&=&P_{k,i}^{p^2}-\gamma_{k,i+1}[u_{k+1}^{p^2}(s_{k+1}+1)]^{p^{2(i-1)}}P_{k,i-1}\\
&=& P_{k,i}^{p^2}-[\gamma_{k,i+1}(s_{k+1}^{p^{2(i-1)}}+1)]u_{k+1}^{p^{2i}}P_{k,i-1}.
\end{array}
$$
Thus for $i=2$ we have
$$
\begin{array}{lll}
u_{k+1}^{p^4}P_{k+1,2}&=& [u_{k+1}^{p^2}P_{k+1,1}]^{p^2}-[\gamma_{k,3}(s_{k+1}^{p^2}+1)]u_{k+1}^{p^4+1}\\
&=& u_{k+1}^{p^4}[P_{k+1,1}^{p^2}-\gamma_{k,3}(s_{k+1}^{p^2}+1)u_{k+1}].
\end{array}
$$
For $i\ge 3$ we have
$$
\begin{array}{lll}
u_{k+1}^{p^{2i}}P_{k+1,i}&=&
[u_{k+1}^{p^{2(i-1)}}P_{k+1,i-1}]^{p^2}-[\gamma_{k,i+1}(s_{k+1}^{p^{2(i-1)}}+1)]u_{k+1}^{p^{2i}+p^{2(i-2)}}P_{k+1,i-2}\\
&=& u_{k+1}^{p^{2i}}[P_{k+1,i-1}^{p^2}-\gamma_{k,i+1}(s_{k+1}^{p^{2(i-1)}}+1)u_{k+1}^{p^{2(i-2)}}P_{k+1,i-2}].
\end{array}
$$

For $i\geq 2$ set $\gamma_{k+1,i}=\gamma_{k,i+1}(s_{k+1}^{p^{2(i-1)}}+1)$. We have that $\gamma_{k+1,i}\in R_{k+1}$
 and $\gamma_{k+1,i}\equiv 1\mbox{ mod }m_{k+1}^2$. Therefore C(3,k+1) holds.

 \vskip .2truein

By induction on $i$ (using C(3,k)) we see that for $i\geq 2$ the largest power of $u_{k+1}$ which divides $P_{k,i}$ in $R_{k+1}$ is $p^{2(i-1)}$. It follows from C(2,k) and Remark \ref{gs} that C(2,k+1) holds.
 \end{proof}

 \vskip .2truein

\begin{Corollary}
With notation as above, for $k\ge 0$ and $i\ge 2$, there exist $\Lambda_{k,i}\in R_k$ such that
\begin{equation}\label{eq4}
P_{k,i}=\left\{\begin{array}{ll}
P_{k,1}^{p^2}-P_{k,0}+\Lambda_{k,2}&\mbox{ if }i=2\\
P_{k,i-1}^{p^2}-P_{k,0}^{p^{2(i-2)}}P_{k,i-2}+\Lambda_{k,i}&\mbox{ if }i\ge 3
\end{array}\right.
\end{equation}
where
\begin{equation}\label{eq5}
\nu(\Lambda_{k,i})\ge \sum_{j=1}^{i-1}p^{4j-2i-2k}+p^{4-2i-2k}.
\end{equation}
\end{Corollary}
\begin{proof}
We use induction on $k$. The statements are true for $k=0$, where $\Lambda_{0,i}=0$. Assume that $k\geq 1$. Using equations (\ref{eq2}) and  (\ref{eq3}) and the induction hypothesis we have that
$$
\begin{array}{lll}
P_{k,i}&=&\dfrac{P_{k-1,i+1}}{u_k^{p^{2i}}}=\dfrac{P_{k-1,i}^{p^2}-P_{k-1,0}^{p^{2i-2}}P_{k-1,i-1}+\Lambda_{k-1,i+1}}{u_k^{p^{2i}}}\\
&=& \dfrac{P_{k-1,i}^{p^2}-(P_{k-1,1}^{p^2}+\Lambda_{k-1,2}-P_{k-1,2})^{p^{2i-2}}P_{k-1,i-1}+\Lambda_{k-1,i+1}}{u_k^{p^{2i}}}\\
&=&
P_{k,i-1}^{p^2}+\dfrac{(u_k^{p^{2i}}v_k^{p^{2i-2}}-u_k^{p^{2i}}-\Lambda_{k-1,2}^{p^{2i-2}})P_{k-1,i-1}}{u_k^{p^{2i}}}
+\dfrac{\Lambda_{k-1,i+1}}{u_k^{p^{2i}}}\\
&=&
P_{k,i-1}^{p^2}-P_{k-1,i-1}+v_k^{p^{2i-2}}P_{k-1,i-1}-\dfrac{\Lambda_{k-1,2}^{p^{2i-2}}P_{k-1,i-1}}{u_k^{p^{2i}}}
+\dfrac{\Lambda_{k-1,i+1}}{u_k^{p^{2i}}}.
\end{array}
$$
Thus we have that
$$
P_{k,2}=P_{k,1}^{p^2}-u_k+u_kv_k^{p^2}-\frac{\Lambda^{p^2}_{k-1,2}P_{k-1,1}}{u_k^{p^4}}+\frac{\Lambda_{k-1,3}}{u_k^{p^4}},
$$
and for $i\ge 3$,
$$
\begin{array}{lll}
P_{k,i}
&=& P_{k,i-1}^{p^2}-P_{k,i-2}u_k^{p^{2(i-2)}}+
u_k^{p^{2(i-2)}}v_k^{p^{2i-2}}P_{k,i-2}-\dfrac{\Lambda_{k-1,2}^{p^{2i-2}}P_{k-1,i-1}}{u_k^{p^{2i}}}
+\dfrac{\Lambda_{k-1,i+1}}{u_k^{p^{2i}}}.
\end{array}
$$

Set \begin{equation}\label{eq6}
\Lambda_{k,2}=u_kv_k^{p^{2}}-\frac{\Lambda_{k-1,2}^{p^{2}}P_{k-1,1}}{u_k^{p^{4}}}
+\frac{\Lambda_{k-1,3}}{u_k^{p^{4}}},
\end{equation}
and for $i\ge 3$,
\begin{equation}\label{eq66}
\Lambda_{k,i}=u_k^{p^{2(i-2)}}v_k^{p^{2i-2}}P_{k,i-2}-\frac{\Lambda_{k-1,2}^{p^{2i-2}}P_{k-1,i-1}}{u_k^{p^{2i}}}
+\frac{\Lambda_{k-1,i+1}}{u_k^{p^{2i}}}.
\end{equation}
In order to verify equation (\ref{eq5}), it suffices to show that the values of each of the three terms in equations (\ref{eq6}) and (\ref{eq66}) satisfy that bound. We use equation (\ref{eq1}) and the induction hypothesis. We have
$$
\nu(u_kv_k^{p^2})=2p^{-2k}.
$$
If $i\ge 3$,
$$
\begin{array}{lll}
\nu(u_k^{p^{2(i-2)}}v_k^{p^{2i-2}}P_{k,i-2})&=& 2p^{2i-2k-4}+\sum_{j=0}^{i-3}p^{4j-2i+4-2k}\\
&=& 2p^{2i-2k-4}+\sum_{j=1}^{i-2}p^{4j-2i-2k}\\
&=& \sum_{j=1}^{i-1}p^{4j-2i-2k}+p^{2i-2k-4}\\
&\ge& \sum_{j=1}^{i-1}p^{4j-2i-2k}+p^{4-2i-2k}.
\end{array}
$$
For $i\geq 2$,
$$
\begin{array}{lll}
\nu\left(\dfrac{\Lambda_{k-1,2}^{p^{2i-2}}P_{k-1,i-1}}{u_k^{p^{2i}}}\right)&=&
\nu(\Lambda_{k-1,2}^{p^{2i-2}})+\nu(P_{k-1,i-1})-p^{2i}\nu(u_k)\\
&\ge&
p^{2i-2}(2p^{-2k+2})+\sum_{j=0}^{i-2}p^{4j-2i-2k+4} -p^{2i-2k}\\
&=& p^{2i-2k}+\sum_{j=1}^{i-1}p^{4j-2i-2k}\\
&\ge& p^{4-2i-2k}+\sum_{j=1}^{i-1}p^{4j-2i-2k},
\end{array}
$$
and
$$
\begin{array}{lll}
\nu\left(\dfrac{\Lambda_{k-1,i+1}}{u_k^{p^{2i}}}\right)
&\ge& \sum_{j=1}^ip^{4j-2i-2k}+p^{2i-2k}+p^{4-2i-2k}-p^{2i-2k}\\
&\ge& \sum_{j=1}^{i-1}p^{4j-2i-2k}+p^{4-2i-2k}.
\end{array}
$$
This completes the proof.
\end{proof}

\vskip .2truein

Set
\begin{equation*}
\Omega=\sum_{i=0}^{\infty}\left(\frac{1}{p^4}\right)^i=\frac{p^4}{p^4-1}.
\end{equation*}

Consider the Artin-Schreier extension $K_1=K(x)$ of $K$. We have that  $x$ is defined by
$$
u=\frac{x^p}{1-x^{p-1}}.
$$
Recall that $\nu_1$ is
the extension of $\nu$ to $K_1$. Write
\begin{equation}
\label{eq10}
x^p=u+h,
\end{equation}
where $\nu_1(h)=1+\dfrac{p-1}{p}=2-\dfrac{1}{p}$.
We have
\begin{equation*}
\left( 2-\frac{1}{p} \right)-\frac{p^4}{p^4-1}=\frac{p^4(p-1)-2p+1}{p(p^4-1)}>\frac{p^4-2p}{p(p^4-1)}=\frac{p^3-2}{(p^4-1)}>0,
\end{equation*}
since $p\ge 2$. Thus
\begin{equation}\label{eq8}
\nu_1(h)>\Omega.
\end{equation}
We further have
\begin{equation}\label{eq9}
-\frac{2}{p}+\frac{1}{p}\Omega<-\frac{1}{p^2},
\end{equation}
since $\dfrac{1}{p^2}-\dfrac{2}{p}+\dfrac{1}{p}\Omega=-\dfrac{1}{p}(\nu_1(h)-\Omega)$.

\vskip .2truein

\vskip .2truein

\begin{Theorem}\label{Theorem1} With the above notation, for $k\ge 0$, we have the following:
\begin{enumerate}
\item[1)] There exists $h_k\in R_k$ and  $\Theta_k\in K_1=K(x)$ such that
$$
h_k^p-x^p=(-1)^kP_{k,k+2}+\Theta_k
$$
where $\nu_1(\Theta_k)>\Omega$.
Thus
 $$
\nu_1(h_k^p-x^p)=1+\frac{1}{p^4}+\cdots+\frac{1}{p^{4k}}+\frac{1}{p^{4(k+1)}}.
$$
\item[2)] If $g\in R_k$ then
$$
\nu_1(g^p-x^p)\le 1+\frac{1}{p^4}+\cdots+\frac{1}{p^{4k}}+\frac{1}{p^{4(k+1)}}.
$$
\end{enumerate}
\end{Theorem}

\begin{proof} We first prove statement 1), using induction on $k$. In the case $k=0$, we observe from (\ref{eq10}) and (\ref{eq8}) that $x^p=u+h$ with  $\nu_1(h)>\Omega$. Set $h_0=P_{0,1}^p$ and $\Theta_0=-h$.
Then
$$
h_0^p-x^p=P_{0,1}^{p^2}-u+\Theta_0=P_{0,2}+\Theta_0.
$$

Now assume that there exists $h_k\in R_k$ such that

$$
h_k^p-x^p=(-1)^kP_{k,k+2}+\Theta_k
$$
with $\nu_1(\Theta_k)>\Omega$. Using equations (\ref{eq3}) and (\ref{eq4}) we have
$$
\begin{array}{lll}
P_{k,k+2}&=&u_{k+1}^{p^{2(k+1)}}P_{k+1,k+1}\\
&=& P_{k+1,0}^{p^{2(k+1)}}P_{k+1,k+1}\\
&=&-P_{k+1,k+3}+P_{k+1,k+2}^{p^2}+\Lambda_{k+1,k+3}.
\end{array}
$$
Set
$h_{k+1}=h_k+(-1)^{k+1}P_{k+1,k+2}^p$. Then
$$
\begin{array}{lll}
h_{k+1}^p-x^p&=&(-1)^{k+1}P_{k+1,k+2}^{p^2}+(-1)^{k+1}P_{k+1,k+3}+(-1)^kP_{k+1,k+2}^{p^2}+(-1)^k\Lambda_{k+1,k+3}+\Theta_k\\
&=&(-1)^{k+1}P_{k+1,k+3}+(-1)^k\Lambda_{k+1,k+3}+\Theta_k.
\end{array}
$$
Set $\Theta_{k+1}=(-1)^k\Lambda_{k+1,k+3}+\Theta_k$. In order to show that $\nu_1(\Theta_{k+1})>\Omega$, we must show that
$\nu_1(\Lambda_{k+1,k+3})>\Omega$.
By equation (\ref{eq5}) we have
$$
\begin{array}{lll}
\nu_1(\Lambda_{k+1,k+3})&\ge& \sum_{j=1}^{k+2}p^{4j-4k-8}+p^{-4k-4}\\
&=& p^{-4k-4}\left(\sum_{j=0}^{k+1}(p^{4})^j+1\right)\\
&=&p^{-4k-4}\left(\dfrac{1-(p^4)^{k+2}}{1-p^4}+1\right)\\
&=& \dfrac{p^{4(k+2)}+p^4-2}{p^{4k+4}(p^4-1)}.
\end{array}
$$
Thus
$$
\nu_1(\Lambda_{k+1,k+3})-\Omega\ge\frac{p^4-2}{p^{4k+4}(p^4-1)}>0.
$$

\vskip .2truein
We now prove statement 2). Fix $k\geq 0$. Suppose, by way of contradiction, that there exists $g\in R_k$ such that
$$
\nu_1(g^p-x^p)>1+\frac{1}{p^4}+\cdots+\frac{1}{p^{4(k+1)}}.
$$
Set $g'=g-h_k$. Then
$$
g^p-x^p= (g')^p+(h_k^p-x^p)=(g')^p+(-1)^kP_{k,k+2}+\Theta_k.
$$
Now, since $\{P_{k,i}\}_{i\geq 0}$ is a generating sequence in $R_k$, we may  write
$g'=\lambda G+H$, where $0\ne \lambda\in {\bf{k}}$, $\nu_1(g')=\nu_1(G)$, $\nu_1(H)>\nu_1(G)$ and
$$
G=u_k^mP_{k,1}^{a_1}P_{k,2}^{a_2}\cdots P_{k,n}^{a_n}
$$
with $m,n, a_i\in\NN_0$ and $a_i<p^2$ for all $i$ (see Definition \ref{genseq} and \cite[Lemma 5.1]{GK}).
Set
$$
\alpha=\nu_1(P_{k,k+2})=\sum_{j=0}^{k+1}p^{4j-4k-4}
$$
 and $\beta=\nu_1(G^p)$. We have $\alpha=\beta$, otherwise $\nu_1(g^p-x^p)\leq \alpha$.
By equation (\ref{eq1}) we have
$$
\beta= mp^{-2k+1}+\sum_{i=1}^na_i\left(\sum_{j=0}^{i-1}p^{4j-2i-2k+1}\right).
$$
Notice that if $a_i=0$ for all $i$, then $\beta=mp^{-2k+1}\neq \alpha$. Therefore we may suppose that $a_n\ne 0$. Then we can write
$$
\beta=a_np^{-2n-2k+1}+c_1p^{-2n-2k+3}
$$
for some $c_1\in \ZZ$ and
$$
\alpha=p^{-4(k+1)}+b_1p^{-4k}
$$
for some $b_1\in \ZZ$.  Notice that, if $2n+2k-3\geq 4k+4$, the equality $\alpha=\beta$ yields that $a_n$ is a multiple of $p^2$, a contradiction since $0<a_n<p^2$.

We thus have
$2n+2k-3<4k+4$; that is, $n\le k+3$. By a similar argument we must have $n=k+3$ and $a_{k+3}=p$. Thus, using equation (\ref{eq1})
$$
\begin{array}{lll}
0&\le& \beta-p\left(\nu_1(P_{k,k+3})\right)\\
&=& \beta-p\left(\sum_{j=0}^{k+2}p^{4j-4k-5}\right)\\
&=& \alpha-p\left(\sum_{j=0}^{k+2}p^{4j-4k-5}\right)\\
&=& \sum_{j=0}^{k+1}p^{4j-4k-4}-\sum_{j=0}^{k+2}p^{4j-4k-4}\\
&=& -p^4<0
\end{array}
$$
giving a contradiction.
\end{proof}

\medskip

\begin{Theorem}\label{Theorem2} With the above notation, suppose that $f\in K$. Then
$$
\nu_1\left(\frac{1}{x}-f\right)<-\frac{2}{p}+\frac{1}{p}\Omega<-\frac{1}{p^2}.
$$
\end{Theorem}

\begin{proof} We may assume that $f\in K$ is such that $\nu_1\left(\dfrac{1}{x}-f\right)>-\dfrac{1}{p}$. Then $\nu_1(f)=\nu_1\left( \dfrac{1}{x} \right)=-\dfrac{1}{p}$.
Set $g=\dfrac{1}{f}$. We have that $\nu(g)=\dfrac{1}{p}>0$, so $g\in V$ (the valuation ring of $\nu$), and thus $g\in R_k$ for some $k\ge 0$. Now we compute,
using Theorem \ref{Theorem1} and equation (\ref{eq9}),
$$
\begin{array}{lll}
\nu_1\left(\dfrac{1}{x}-f\right)&=&\nu_1\left(\dfrac{1}{x}-\dfrac{1}{g}\right)\\
&=&\nu_1\left(\dfrac{g-x}{xg}\right)=\nu_1(g-x)-\nu_1(xg)\\
&=&\dfrac{1}{p}\nu_1(g^p-x^p)-\dfrac{2}{p}\\
&\le& \dfrac{1}{p}\left(1+\dfrac{1}{p^4}+\cdots+\dfrac{1}{p^{4(k+1)}}\right)-\dfrac{2}{p}\\
&<& \dfrac{1}{p}\Omega-\dfrac{2}{p}<-\dfrac{1}{p^2}.
\end{array}
$$
\end{proof}

\begin{Theorem}\label{maintheorem} Let ${\bf k}$ be an algebraically closed field of characteristic $p>0$, let $K={\bf k}(u,v)$ be an algebraic function field, let $K_1=K(x)$, where $u=x^p/(1-x^{p-1})$.
Let $\nu$ be the valuation on $K$ determined by the generating sequence
$P_0=u$, $P_1=v$, $P_2=v^{p^2}-u$  and
$P_{i+1}=P_i^{p^2}-u^{p^{2i-2}}P_{i-1}$ for $i\ge 2$ in $R={\bf k}[u,v]_{(u,v)}$. Let $\nu_1$ be the unique extension of $\nu$ to $K_1$. Then $(K_1,\nu_1)$ is a dependent Artin Schreier defect extension of $(K, \nu)$.
\end{Theorem}

\begin{proof} Recall that $\frac{1}{x}$ is an Artin-Schreier generator of $K_1$ over $K$. The conclusion now follows from Theorem \ref{thcriteria} and Theorem \ref{Theorem2}.
\end{proof}

\end{document}